\documentclass{amsart}

\usepackage{amsmath,arydshln,multirow}

\usepackage{hyperref}

\usepackage{cases}
\usepackage{amsmath}
\usepackage{amsfonts}
\usepackage{bm}
\usepackage{arydshln}
\usepackage{amsfonts,amsmath,amssymb,amscd,bbm,amsthm,mathrsfs,dsfont}
\usepackage{mathrsfs}
\usepackage{pb-diagram}
\usepackage{amssymb}

\newtheorem{Thm}{Theorem}[section]

\newtheorem{Lem}[Thm]{Lemma}

\newtheorem{Def}[Thm]{Definition}

\newtheorem{proposition-definition}[Thm]{Proposition-Definition}

\newtheorem{conj}[Thm]{Conjecture}

\theoremstyle{remark}
\newtheorem{Rem}[Thm]{Remark}

\numberwithin{equation}{section}

\begin{document}

\title{A finiteness theorem for mod $p$ Galois representations over global function fields}


	\author{Yufan Luo}
	
	\subjclass[2020]{}
	
	\dedicatory{}
	\address{Shanghai Institute for Mathematics and Interdisciplinary Sciences (SIMIS), Shanghai 200433, China
	}
	
	\address{Research Institute of Intelligent Complex Systems, Fudan University, Shanghai
		200433, China}
	\email{yufanluo@hotmail.com}

	\maketitle
	\begin{abstract}
		Let $p$ be an odd prime number and let $\overline{\mathbb{F}}_p$ be a fixed algebraic closure of the finite field of order $p$. Let $K$ be a global function field of characteristic different from $p$ and let $G_{K}$ be the absolute Galois group of $K$. We prove that there are only finitely many isomorphism classes of continuous geometric semisimple representations $\rho:G_{K}\to \mathrm{GL}_{n}(\overline{\mathbb{F}}_{p})$ such that their Artin conductors are bounded. It is worth emphasizing that we do not need to assume that $p$ does not divide $n$.
	\end{abstract}

	\section{Introduction}
	Let $p$ be a prime number and let $\overline{\mathbb{F}}_p$ denote an algebraic closure of the finite field of order $p$. The following conjecture was independently proposed by Khare \cite{MR1751924} and Moon \cite{MR1782427}:

	\begin{conj}
	Let $K$ be a number field, $n$ be a positive integer and $\mathfrak{R}$ be a nonzero ideal of the ring of integers of $K$. Let $G_{K}$ be the absolute Galois group of $K$. Then there are only finitely many isomorphism classes of continuous semisimple representations 
	\[ \rho:G_{K}\to \mathrm{GL}_{n}(\overline{\mathbb{F}}_{p}) \]
	such that the prime-to-$p$ Artin conductor $\mathfrak{R}(\rho)$ of $\rho$ is bounded by $\mathfrak{R}$. 
	\end{conj}
	
	In \cite[Theorem 2]{MR1838373}, Moon and Taguchi established the conjecture under the assumption that the image of $\rho$ is solvable. They also formulated the following function field analogue of this conjecture:
	
	\begin{conj}\label{Kharemoonconjectre}
		Let $K$ be a global function field of characteristic different from $p$ with absolute Galois group $G_{K}$. Let $n$ be a positive integer and $\mathfrak{R}$ an effective divisor of $K$. Then there are only finitely many isomorphism classes of continuous semisimple representations 
		\[
		\rho: G_{K} \to \mathrm{GL}_{n}(\overline{\mathbb{F}}_{p})
		\]
		whose Artin conductor $\mathfrak{R}(\rho)$ divides $\mathfrak{R}$, and such that $\rho$ is geometric, i.e., the fixed field of the kernel of $\rho$ contains no constant field extension of $K$.
	\end{conj}

	In \cite[Theorem 4]{MR1838373}, Moon and Taguchi established this conjecture provided that the image of $\rho$ is solvable. Furthermore, Böckle and Khare \cite[Corollary 1.7 and Remark 1.8]{MR2218896} verified the conjecture under the assumptions that $\mathrm{im}(\rho)$ is large, $p \nmid n$, and $\rho$ is everywhere tamely ramified. Additionally, assuming $p \nmid n$, an affirmative answer is implicit in the proof of \cite[Proposition E.10.1]{MR3762002}, as well as in the proofs of Lemma 6.4 and Theorem 6.8 in \cite{MR4014290}.
	
   The main result of this paper is as follows, which removes the restriction that $p \nmid n$.
	
	\begin{Thm}\label{main}
	Conjecture \ref{Kharemoonconjectre} holds provided that $p$ is odd or $n=2$.
	\end{Thm}
    
    Our proof is based on de Jong’s conjecture \cite{MR1818381} and its solution, which was established by Gaitsgory for $p>2$ in \cite{MR2342444} and by de Jong for $n \leq 2$. It is worth noting that to remove the assumption that $p \nmid n$, a further refinement of the techniques in \cite{MR1818381} is required. By replacing the original methods with a computation of the universal deformation ring via the global Euler--Poincaré characteristic, we are able to remove the assumption that $p \nmid n$.
    
  Throughout this paper, $p$ is a rational prime, $\mathbb{F}_{p}$ denotes the finite field of order $p$, and $\mathbb{Q}_{p}$ is the field of $p$-adic numbers. We fix algebraic closures $\overline{\mathbb{F}}_{p}$ and $\overline{\mathbb{Q}}_{p}$ of $\mathbb{F}_{p}$ and $\mathbb{Q}_{p}$, respectively.
    
	\section{Proof of Theorem \ref{main}}
	\subsection{A conjecture of de Jong}
	Let $k$ be a finite field of characteristic $\ell$ with $\ell\neq p$. Let $X$ be a geometrically connected, smooth curve over $k$. Let $\pi_{1}(X)$ denote the arithmetic fundamental group of $X$. Then we have the following exact sequence of profinite groups:
	\[ 1\to \pi_{1}(\overline{X})\to \pi_{1}(X)\to G_{k}\to 1,\]
	where $\overline{X}$ denotes the base change of $X$ to an algebraic closure of $k$ and $G_{k}$ denotes the absolute Galois group of $k$. (As usual, base points are omitted from the notation.)
	
	The following celebrated conjecture was proposed by de Jong in \cite{MR1818381}:
	\begin{conj}\label{dejong}
	We keep the notation as above. Suppose that $\rho:\pi_{1}(X)\to \mathrm{GL}_{n}(F)$ is a continuous representation of $\pi_{1}(X)$ where $F$ is a local field of characteristic $p$. Then $\rho(\pi_{1}(\overline{X}))$ is finite.
	\end{conj}
	
	\begin{Thm}\label{dejonggait}
		Conjecture \ref{dejong} holds provided that $p$ is odd or $n=2$.
	\end{Thm}
	\begin{proof}
	If $n=2$, then Conjecture \ref{dejong} holds by \cite[Corollary 4.10]{MR1818381}. Now suppose that $p>2$. By \cite[Proposition 2.13]{MR1818381}, we may assume that $X$ is projective. Then our claim follows from \cite{MR2342444}.
	\end{proof}
	\begin{Rem}
		Böckle and Khare independently verified the conjecture in the case where the associated residual representation has large image (see \cite{MR2218896}).
	\end{Rem}
	
 Let $\mathbb{F}$ be a finite field of characteristic $p$ and $W(\mathbb{F})$ be the Witt vector ring of $\mathbb{F}$. Let $\mathcal{C}$ be the category of complete Noetherian local $W(\mathbb{F})$-algebras $R$ with residue field $R/\mathfrak{m}_{R}=\mathbb{F}$ where $\mathfrak{m}_{R}$ denotes the maximal ideal of $R$. 
 
 Suppose that $\overline{\rho}:\pi_{1}(X)\to \mathrm{GL}_{n}(\mathbb{F})$ is an absolutely irreducible continuous representation, and that $\eta:\pi_{1}(X)\to \mathrm{GL}_{1}(W(\mathbb{F}))=W(\mathbb{F})^{\times}$ is a continuous character such that $\eta \equiv \det(\overline{\rho}) \pmod p$. Consider the functor from $\mathcal{C}$ to the category of sets which maps $R$ to the set of equivalence classes of continuous representations $\rho_{R}:\pi_{1}(X)\to \mathrm{GL}_{n}(R)$ such that $\rho_{R}\pmod{\mathfrak{m}_{R}} = \overline{\rho}$ and $\det(\rho_{R})=(W(\mathbb{F})^{\times}\to R^{\times})\circ \eta$. Here, we say that two representations $\rho_{R}$ and $\rho_{R}'$ are equivalent if $\rho_{R}=B \rho'_{R}B^{-1}$ for some matrix $B\in \mathrm{GL}_{n}(R)$ congruent to the identity matrix modulo $\mathfrak{m}_{R}$. Mazur proved that this functor is representable by an object $R_{\overline{\rho}}^{\eta}\in \mathcal{C}$ (see \cite{MR1012172}, \cite[Section 3]{MR1818381}, and \cite[Proposition 5.1.1]{MR3184335}).
 
	\begin{Lem}\label{refinemnetlemma}
		We keep the notation as above. Assume that $X$ is affine. Then the ring $R_{\overline{\rho}}^{\eta}$ is isomorphic to 
		\[ W(\mathbb{F})[[X_{1},\cdots,X_{s}]]/(f_{1},\cdots,f_{s}),\]
		for some integer $s$.
	\end{Lem}
	\begin{proof}
		Let $\operatorname{ad}(\overline{\rho})$ denote the adjoint representation of $\overline{\rho}$ and let $\operatorname{ad}^{0}(\overline{\rho})$ be the subrepresentation on the trace zero matrices. We put $h_{i}:=\dim_{\mathbb{F}}H^{i}(\pi_{1}(X),\operatorname{ad}^{0}(\overline{\rho}))$. By \cite[Proposition 5.1.1]{MR3184335}, the ring $R_{\overline{\rho}}^{\eta}$ is isomorphic to 
		\[ W(\mathbb{F})[[X_{1},\cdots,X_{m}]]/(f_{1},\cdots,f_{h_{2}}),\]
		where $m$ is the dimension of the dual of the mod $p$ tangent space of $R_{\overline{\rho}}^{\eta}$. That is,
		\[ m = \dim_{\mathbb{F}}Z^1( \operatorname{ad}^{0}(\overline{\rho})) - \dim_{\mathbb{F}} B^1(\operatorname{ad}(\overline{\rho})), \]
		where $Z^{1}$ (resp. $B^1$) denotes the $\mathbb{F}$-vector space of $1$-cocycles (resp. the $\mathbb{F}$-vector space of $1$-coboundaries). 
		
		Since $h_{1} = \dim_{\mathbb{F}} Z^1(\operatorname{ad}^{0}(\overline{\rho})) - \dim_{\mathbb{F}} B^1(\operatorname{ad}^{0}(\overline{\rho}))$, we obtain that 
		 \[ m = h_{1} - \left( \dim_{\mathbb{F}} B^1(\operatorname{ad}(\overline{\rho})) - \dim_{\mathbb{F}} B^1(\operatorname{ad}^{0}(\overline{\rho})) \right). \]
	   Since 
		\[ \dim_{\mathbb{F}} B^1(\operatorname{ad}(\overline{\rho})) - \dim_{\mathbb{F}} B^1(\operatorname{ad}^{0}(\overline{\rho}))= (n^2 - 1) - [(n^2 - 1) - h_{0}] = h_{0}, \]
	  we conclude that $m = h_{1} - h_{0}$. Since $X$ is affine, by the global Euler--Poincar\'{e} characteristic \cite[Chapter I, Theorem 5.1]{MR2261462}, we have $h_{1}-h_{0}=h_{2}$. By setting $s = h_{2}$, the claim follows.
	\end{proof}

	We prove the following theorem, which generalizes the result \cite[Theorem 3.5]{MR1818381} of de Jong by removing the assumption that $p$ does not divide $n$.

	\begin{Thm}\label{lifting}
		We keep the notation as above. Assume that
		\begin{enumerate}
			\item $X$ is affine;
			\item Conjecture \ref{dejong} holds for $X$ and $n$;
			\item The restriction $\overline{\rho}|_{\pi_{1}(\overline{X})}$ is absolutely irreducible.
		\end{enumerate}
		 Then the morphism $W(\mathbb{F})\to R_{\overline{\rho}}^{\eta}$ is a finite flat complete intersection morphism. In particular, $\overline{\rho}$ can be lifted to a representation $\rho:\pi_{1}(X)\to \mathrm{GL}_{n}(\overline{\mathbb{Q}}_{p})$.
	\end{Thm}
	\begin{proof}
		The result follows from Lemma \ref{refinemnetlemma} by applying the argument used in \cite[Section 3.14]{MR1818381}.
	\end{proof}
	
	\subsection{Ramification theory}
	In this subsection, we review the theory of ramification, for which our primary reference is \cite{KR14}. Let $K$ be a function field of characteristic $\ell\neq p$ and $G_{K} $ be the absolute Galois group of $K$. For every prime divisor $\mathfrak{q}$ of $K$, let $G_{\mathfrak{q}} \subset G_K$ be a decomposition group at $\mathfrak{q}$, and let $(G_{\mathfrak{q}}^{\lambda})_{\lambda\in \mathbb{R}_{\geq -1}}$ be the ramification filtration of $G_{\mathfrak{q}}$ in the upper numbering and $G_{\mathfrak{q}}^{\lambda+}$ be the topological closure of $\bigcup_{\lambda'>\lambda}G_{\mathfrak{q}}^{\lambda'}$ in $G_{\mathfrak{q}}$ for $\lambda\in \mathbb{R}_{\geq 0}$. We recall the following definitions.

\begin{Def}[{\cite[Definition 4.82]{KR14}; see also \cite[Section 3]{MR3058662}}]
Suppose that $\rho: G_{K} \to \mathrm{GL}(V)$ is a continuous representation on a finite-dimensional vector space $V$ over $\overline{\mathbb{Q}}_{p}$. The \textit{Swan conductor} of $\rho$ is defined as:
\[ \operatorname{Sw}(\rho) := \prod_{\mathfrak{q}} \mathfrak{q}^{n_{0}(\mathfrak{q})},  \]
where $\mathfrak{q}$ runs over the prime divisors of $K$ and, for every $\mathfrak{q}$,
\[ n_{0}(\mathfrak{q}) :=\sum_{\lambda>0}\lambda \dim(V^{G_{\mathfrak{q}}^{\lambda+}}/V^{G_{\mathfrak{q}}^{\lambda}}), \]
where $V^{G_{\mathfrak{q}}^{\lambda+}}$ and $V^{G_{\mathfrak{q}}^{\lambda}}$ are the subspaces of $V$ fixed by $G_{\mathfrak{q}}^{\lambda+}$ and $G_{\mathfrak{q}}^{\lambda}$ respectively.
\end{Def}

\begin{Def}[{\cite[Section 4]{MR1838373}; see also \cite[Section 4.4]{KR14}}]
Suppose that $\overline{\rho}: G_{K} \to \mathrm{GL}(V)$ is a continuous representation on a finite-dimensional vector space $V$ over $\overline{\mathbb{F}}_{p}$. Let $L/K$ be a finite Galois extension such that $\overline{\rho}$ factors through $\operatorname{Gal}(L/K)$. The \textit{Swan conductor} (resp. \textit{Artin conductor}) of $\overline{\rho}$ is defined as:
\[ \operatorname{Sw}(\overline{\rho}) := \prod_{\mathfrak{q}} \mathfrak{q}^{n_{0}(\mathfrak{q})}, \quad \left( \text{resp. } \mathfrak{R}(\overline{\rho}) := \prod_{\mathfrak{q}} \mathfrak{q}^{n(\mathfrak{q})} \right), \]
where $\mathfrak{q}$ runs over the prime divisors of $K$ and, for every $\mathfrak{q}$,
\[ n_{0}(\mathfrak{q}) := \sum_{i=1}^{\infty} \frac{1}{(G_{0}(\mathfrak{q}):G_{i}(\mathfrak{q}))} \dim (V/V^{G_{i}(\mathfrak{q})}), \]
\[ \left( \text{resp. } n(\mathfrak{q}) := \sum_{i=0}^{\infty} \frac{1}{(G_{0}(\mathfrak{q}):G_{i}(\mathfrak{q}))} \dim (V/V^{G_{i}(\mathfrak{q})}) \right), \]
where $G_{i}(\mathfrak{q})$ denotes the $i$-th ramification subgroup in the lower numbering of the decomposition group of a prime of $L$ above $\mathfrak{q}$ and $V^{G_{i}(\mathfrak{q})}$ denotes the subspace of $V$ fixed by $G_{i}(\mathfrak{q})$.
\end{Def}

We also recall the following well-known compatibility of the Swan conductor with respect to the residual reduction.

\begin{Lem}\label{swanconductorisinvariant}
	Let $K$ be a function field of characteristic $\ell \neq p$. Suppose that $\rho: G_{K} \to \mathrm{GL}_n(\overline{\mathbb{Q}}_{p})$ is a continuous representation, and let $\overline{\rho}: G_{K} \to \mathrm{GL}_n(\overline{\mathbb{F}}_{p})$ be any residual representation of $\rho$. Then we have
	\[  \operatorname{Sw}(\rho) = \operatorname{Sw}(\overline{\rho}).\]
\end{Lem}
\begin{proof}
	This assertion follows immediately from \cite[Proposition 4.74 and Theorem 4.86]{KR14}.
\end{proof}

\subsection{Proof of main theorem}
	The following finiteness result is a consequence of Lafforgue's work \cite{MR1875184} on the global Langlands correspondence for function fields and of results by Harder, Gelfand, and Piatetski-Shapiro (see \cite[Section 4.3]{MR3058662}).
	
	\begin{Thm}\label{Lafforgue}
		Let $X$ be a geometrically connected, smooth curve over $k$ with function field $K$. Let $n$ be a positive integer, let $\mathfrak{R}$ be an effective divisor of $K$, and let $\chi:\pi_{1}(X) \to \overline{\mathbb{Q}}_{p}^{\times}$ be a continuous character of finite order. Then there are only finitely many isomorphism classes of continuous irreducible representations
		\[ \rho: \pi_{1}(X) \to \mathrm{GL}_{n}(\overline{\mathbb{Q}}_{p}) \]
		such that $\det(\rho) = \chi$ and the Swan conductor $\operatorname{Sw}(\rho)$ is bounded by $\mathfrak{R}$.
	\end{Thm}
	
We prove the following result, which shows that Conjecture \ref{dejong} implies Conjecture \ref{Kharemoonconjectre}.
	
	\begin{Thm}\label{realmaintheorem}
    Let $n$ be a positive integer and let $X$ be a geometrically connected, smooth curve over $k$ with function field $K$. Assume that Conjecture \ref{dejong} holds for continuous representations of dimension $r \le n$. Let $\mathfrak{R}$ be an effective divisor of $K$. Then there are only finitely many isomorphism classes of continuous semisimple representations 
		\[
		\overline{\rho}: \pi_{1}(X) \to \mathrm{GL}_n(\overline{\mathbb{F}}_p)
		\]
		such that the Artin conductor $\mathfrak{R}(\overline{\rho})$ of $\overline{\rho}$ divides $\mathfrak{R}$ and $\overline{\rho}(\pi_{1}(X)) = \overline{\rho}(\pi_{1}(\overline{X}))$.
	\end{Thm}
	\begin{proof}
		It suffices to prove the theorem in the case where $X$ is affine. Indeed, choose a nonempty affine open subscheme $U\subset X$. The open immersion $U\hookrightarrow X$ induces a natural surjection $\pi_1(U)\twoheadrightarrow \pi_1(X)$. Thus every representation of $\pi_1(X)$ may be regarded, by pullback, as a representation of $\pi_1(U)$ with the same image and bounded Artin conductor. 
		
		 Let $\overline{\rho}: \pi_{1}(X) \to \mathrm{GL}_n(\overline{\mathbb{F}}_p)$ be a continuous semisimple representation satisfying the given conditions. Since $\overline{\mathbb{F}}_p$ is algebraically closed, $\overline{\rho}$ decomposes into a direct sum of absolutely irreducible representations. Since each irreducible component also satisfies the given conditions in its respective dimension $r \leq n$, we may assume that $\overline{\rho}$ is absolutely irreducible. 
		By compactness, $\overline{\rho}$ has finite image and hence we can assume that $\overline{\rho}(\pi_{1}(X)) \subset \mathrm{GL}_{n}(\mathbb{F})$ for some finite subfield $\mathbb{F} \subset \overline{\mathbb{F}}_{p}$. Let $\eta:\pi_{1}(X)\to W(\mathbb{F})^{\times}$ be the Teichmüller lift of $\det(\overline{\rho})$. By our assumption and Theorem \ref{lifting}, we can find a continuous lift $\rho:\pi_{1}(X)\to \mathrm{GL}_{n}(\mathcal{O})$ of $\overline{\rho}$ with $\det(\rho) = \eta$ where $\mathcal{O}$ is the ring of integers of a finite extension of the fraction field of $W(\mathbb{F})$.
		
		Now, Lemma \ref{swanconductorisinvariant} implies that $\operatorname{Sw}(\rho) = \operatorname{Sw}(\overline{\rho})$. 
		Since the Artin conductor bounds the Swan conductor (i.e., $\operatorname{Sw}(\overline{\rho})$ divides $\mathfrak{R}(\overline{\rho})$), it follows that $\operatorname{Sw}(\rho)$ divides $\mathfrak{R}$. Given that there are only finitely many geometric abelian extensions of $K$ with Artin conductor dividing $\mathfrak{R}$ (see \cite[Section 4]{MR1838373}), there are only finitely many possibilities for the determinant $\eta$. Invoking Theorem \ref{Lafforgue}, we conclude that there are only finitely many such lifts $\rho$. Consequently, there are only finitely many such representations $\overline{\rho}$, which completes the proof.
	\end{proof}

	Finally, we prove our main theorem as follows.
	\begin{proof}[Proof of Theorem \ref{main}]
		Let $C$ be the smooth projective curve over the finite field $k$ with function field $K$, and put
		\[
		X=C\setminus \operatorname{Supp}(\mathfrak R),
		\]
			where $\operatorname{Supp}(\mathfrak R)$ denotes the support of the effective divisor $\mathfrak R$. Every representation $\overline{\rho}:G_{K}\to \mathrm{GL}_{n}(\overline{\mathbb{F}}_{p})$ of $G_K$ whose Artin conductor divides $\mathfrak R$ is unramified outside $\operatorname{Supp}(\mathfrak R)$ and hence factors through $\pi_1(X)$. Note also that the condition that the fixed field of the kernel of $\overline{\rho}$ contains no nontrivial constant field extension of $K$ is equivalent to
		\[
		\overline{\rho}(\pi_1(X))=\overline{\rho}(\pi_1(\overline{X})).
		\]
		By Theorem \ref{dejonggait}, Conjecture \ref{dejong} holds in all dimensions when $p$ is odd and in dimension $2$ for arbitrary $p$. Thus, the theorem follows from Theorem \ref{realmaintheorem}. 
	\end{proof}

\section*{Acknowledgments}
The author thanks the anonymous referee for pointing out \cite[Proposition E.10.1]{MR3762002}, as well as Lemma 6.4 and Theorem 6.8 in \cite{MR4014290}.

	\bibliographystyle{plain}
	\bibliography{FM}

\end{document}